\documentclass[10pt, a4paper]{article}
\usepackage{amssymb}
\usepackage{amsmath}
\usepackage{amsfonts}
\usepackage{amscd}
\usepackage{geometry}
\usepackage[all,cmtip]{xy}
\usepackage{makeidx}
\usepackage{titlesec}
\usepackage{fancyhdr}

\newcommand{\C}{\mathbb{C}}
\newcommand{\Q}{\mathbb{Q}}
\newcommand{\N}{\mathbb{N}}
\newcommand{\Z}{\mathbb{Z}}
\newcommand{\A}{\mathbb{A}}
\newcommand{\F}{\mathbb{F}}

\DeclareMathOperator{\proj}{proj}
\DeclareMathOperator{\rec}{rec}
\DeclareMathOperator{\Ind}{Ind}
\DeclareMathOperator{\Gal}{Gal}
\DeclareMathOperator{\GL}{GL}
\DeclareMathOperator{\GSp}{GSp}
\DeclareMathOperator{\PGL}{PGL}
\DeclareMathOperator{\SL}{SL}
\DeclareMathOperator{\SU}{SU}
\DeclareMathOperator{\PGSp}{PGSp}
\DeclareMathOperator{\PSp}{PSp}
\DeclareMathOperator{\Sp}{Sp}
\DeclareMathOperator{\PO}{PO}
\DeclareMathOperator{\SO}{SO}
\DeclareMathOperator{\PSO}{PSO}
\DeclareMathOperator{\PGO}{PGO}
\DeclareMathOperator{\GO}{GO}
\DeclareMathOperator{\POM}{P\Omega}

\titleformat{\section}[hang]
{\normalfont\filright\large}{\thesection. }{0pt}
{\upshape\bfseries}

\titleformat{\subsection}[hang]
{\itshape}{\thesubsection \ - }{0pt}
{}

\usepackage{amsthm}
\theoremstyle{plain}
\newtheorem{theo}{Theorem}[section]

\newtheorem{coro}[theo]{Corollary}
\theoremstyle{definition}

\newtheorem{ejem}[theo]{Example}

\title{Automorphic Galois representations and the inverse Galois problem for certain groups of type $D_{m}$}
\author{\small ADRI\'AN ZENTENO \footnote{Instituto de Matem\'aticas, Pontificia Universidad Cat\'olica de Valpara\'iso, Blanco Viel 596, Cerro Bar\'on, Valpara\'iso, Chile. \texttt{adrian.zenteno@pucv.cl}}}

\date{\today}
\begin{document}
\maketitle

\begin{abstract}
Let $m$ be an integer greater than three and $\ell$ be an odd prime. In this paper, we prove that at least one of the following groups: $\POM^\pm_{2m}(\F_{\ell^s})$, $\PSO^\pm_{2m}(\F_{\ell^s})$, $\PO_{2m}^\pm(\F_{\ell^s})$ or $\PGO^\pm_{2m}(\F_{\ell^s})$ is a Galois group of $\Q$ for infinitely many integers $s > 0$.
This is achieved by making use of a slight modification of a group theory result of Khare, Larsen and Savin, and previous results of the author on the images of the Galois representations attached to cuspidal automorphic representations of $\GL_{2m}(\A_\Q)$.

\end{abstract}

\section{Introduction}

In recent years, the study of the images of the Galois representations associated to RAESDC (regular algebraic, essentially self-dual, cuspidal) automorphic representations of $\GL_n(\A_{\Q})$ with prescribed local conditions has been an effective strategy to address the inverse Galois problem for finite groups of Lie type. In particular, the existence of RAESDC automorphic representations of $\GL_n(\A_{\Q})$, with a local component $\pi_p$ that is a self-dual supercuspidal representation of $\GL_n(\Q_p)$ of depth zero, has been crucial to construct Galois representations with controlled image. 

For example, in \cite{KLS08}, self-dual supercuspidal representations of $\GL_n(\Q_p)$, associated to certain tamely ramified symplectic representations of $\Gal(\overline{\Q}_p /\Q_p)$, were used to show that for any prime $\ell$ there exist infinitely many positive integers $s$ such that either $\PSp_n(\F_{\ell^s})$ or $\PGSp_n(\F_{\ell^s})$ can be realized as a Galois group over $\Q$. Similar results have been obtained in \cite{KLS10} for groups of type $B_{m}$ and $G_2$.

In this paper, we prove a similar result for groups of type $D_m$ by studying the images of certain tamely ramified orthogonal representations of $\Gal(\overline{\Q}_p /\Q_p)$ associated to self-dual supercuspidal representations of $\GL_{2m}(\Q_p)$. More precisely we prove the following result.

\begin{theo}\label{impo}
Let $m \geq 4$ be an integer and $\ell$ be an odd prime. Then, there exist infinitely many
positive integers $s$ such that at least one of the following groups:
\[
\POM^\pm_{2m}(\F_{\ell^s}), \; \PSO^\pm_{2m}(\F_{\ell^s}), \; \PO_{2m}^\pm(\F_{\ell^s}) , \; \PGO^\pm_{2m}(\F_{\ell^s})
\]
can be realized as a Galois group over $\Q$.
\end{theo}

To the best of our knowledge, these orthogonal groups are not previously known to be Galois over $\Q$, except for some cases where $s$ or $m$ is  small, which were studied in \cite{MM99}, \cite{Re99}, \cite{Zi}, \cite{Ze19} and \cite{Ze20}.

\paragraph{Notation:} Through this paper, if $F$ is a perfect field, we denote by $\overline{F}$ an algebraic closure of $F$ and by $G_F$ the absolute Galois group $\Gal(\overline{F}/F)$. Whenever $G$ is a subgroup of a certain linear group $\GL_n(F)$, we write P$G$ for the image of $G$ in the projective linear group $\PGL_n(F)$. In fact, for classical groups in general, we will use the same notation and conventions as in Section 2 of \cite{Ze19}.

\section{Admissible pairs and local Langlands correspondence}\label{se1}

Let $F$ be a $p$-adic field (i.e., a finite extension of the $p$-adic field $\Q_p$) with ring of integers $\mathcal{O}_F$ and Weil group $W_F$.
We denote by $\mathfrak{p}_F$ the maximal ideal of $\mathcal{O}_F$ and by $q$ the order of the residue field $\kappa_F = \mathcal{O}_F/\mathfrak{p}_F$. Moreover, $U_F^1 = 1+ \mathfrak{p}_F$ denotes the group of 1-units and $\varpi$ denotes a uniformizing element. For each integer $n \geq 1$, let $\mathcal{G}_F^0(n)$ be the set of equivalence classes of irreducible smooth representations of $W_F$ of dimension $n$ and $\mathcal{A}_F^0(n)$ be the set of equivalence classes of irreducible admissible supercuspidal representations of $\GL_n(F)$. The local Langlands correspondence gives a bijective map
\[
\rec_{F,n} : \mathcal{G}_F^0(n) \longrightarrow \mathcal{A}_F^0(n) 
\]
for each $n$ \cite{HT01} \cite{He00} \cite{Sch13} . Unfortunately, the existence of the family $\{ \rec_{F,n} \}_n$  has been established indirectly and explicit information about it is very hard to obtain. However, when $p$ and $n$ are relatively primes, Howe \cite{Ho77} defined a set of characters of certain extensions of $F$ of degree $n$ which parameterizes both $\mathcal{G}_F^0(n)$ and $\mathcal{A}_F^0(n)$, and allows us to describe (in this case) the local Langlands correspondence in a very explicit way (see also \cite{Moy86} and \cite{BH05}). So, from now on, we will assume that $(n,p) = 1$.

Let $E/F$ be a tamely ramified extension of degree $n$ and $\chi$ be a character of $E^\times$. The pair $(E/F, \chi)$ is called admissible if it satisfies the following two conditions. Let $L$ range over intermediate fields, $F \subseteq L \subseteq E$.
\begin{enumerate}
\item If $\chi$ factors through the relative norm $N_{E/L}$, then $L = E$.
\item If $\chi \vert_{U^1_E}$, factors through $N_{E/L}$, then $E/L$ is unramified.
\end{enumerate}

Two admissible pairs $(E/F,\chi)$ and $(E'/F' , \chi ')$ are equivalents if there is an $F$-isomorphism $\psi : E \rightarrow E'$ such that $\chi = \chi' \circ \psi$. The map 
\[
(E/F, \chi) \longmapsto \rho_{\chi} = \Ind_{W_E}^{W_F} \chi
\]
provides a canonical bijection between the set $P_n(F)$ of equivalence classes of admissible pairs $(E/F, \chi)$, with $E/F$ of degree $n$, and $\mathcal{G}_F^0(n)$. Here, we regard $\chi$ as a character of $W_E$ via class field theory. Similarly, it is possible to construct a canonical bijection $(E/F, \chi) \mapsto \pi_\chi$ between $P_n(F)$ and $\mathcal{A}_F^0(n)$. For our purposes, it will be enough to know the explicit construction of the irreducible representations of $W_F$. We refer the reader to \cite{Ho77} and \cite{Moy86} for details about the construction of the supercuspidal representations of $\GL_n(F)$ associated to the admissible pairs $(E/F, \chi)$. These two bijections yield a canonical bijection
\[
\rec^{\mathcal{N}}_{F,n} : \mathcal{G}_F^0(n) \longrightarrow \mathcal{A}_F^0(n) 
\]
for each $n$. In \cite[Theorem A]{BH05}, $\rec_{F,n}$ and $\rec^{\mathcal{N}}_{F,n}$ were compared, and it was proven that they differ by a character. More precisely, if $(E/F, \chi) \in P_n(F)$, there is a tamely ramified character $\mu$ of $E^\times$ such that $(E/F, \mu \chi)$ is admissible and
\[
\rec_{F,n}(\rho_\chi)= \pi_{\mu \chi}.  
\]

Of interest for us will be the self-dual representations. In \cite{Ad97}, Adler proved that $\rho_{\chi}$ is self-dual if and only if one of the following conditions holds:
\begin{enumerate}
\item there is an intermediate field $F \subseteq L \subseteq E$, such that $[E:L]=2$ and $\chi \vert _{N_{E/L}(E^\times)}$ is trivial,
\item $p=2$ and $\chi$ has order two.
\end{enumerate}
In particular, this implies, by local Langlands correspondence, that $\GL_n$ admits self-dual supercuspidal representations only if $n$ or $p$ is even.

Finally, and for the sake of the explicitness, we will explain how to construct concrete examples which will be useful through this article.

\begin{ejem}\label{eje1}
First take any even integer $n$ and $E$ the unique unramified extension of $F$ of degree $n$. Recall that $E^{\times} \simeq \langle \varpi \rangle \times \kappa_E^{\times} \times U^1_E$. Then, take any integer $t$ that divides $q^{n/2}+1$ but that does not divide any $q^{n/p_i}-1$, where $p_i$ range over the primes dividing $n$. Finally, take any character of $\kappa_E^{\times}$ of order $t$, inflate trivially to a character of $U_E^1$, and extend to $E^\times$ by sending $\omega$ to either 1 or $-1$. So, we obtain an admissible pair $(E/F, \chi)$ such that $\rho_{\chi}$ is self-dual.
\end{ejem}

We remark that all previous constructions are much more general than we need. However, we decided to write this section in this way since the exposition hardly simplifies by restricting it to the particular case that we need. 

\section{On the images of certain tame self-dual representations of $W_{\Q_p}$}\label{se2}

In this section we will study the image of some of the Galois representations constructed in Example \ref{eje1}.

Let $n$ be a positive even integer and $p>n$ be a prime. Let $t$ be a prime such that $t \equiv 1 \mod n$ and the order of $p$ modulo $t$ is $n$.  
Let $E$ be the unique unramified extension of $\Q_p$ of degree $n$ and $E^{\times} \simeq \langle p \rangle \times \F^{\times}_{p^n} \times U^1_E$.
We will say that a character $\chi : E^{\times} \rightarrow \C^{\times}$ is of $O$-\emph{type} (resp. $S$-\emph{type}) \emph{at $p$ of order $t$}, if $\chi (p) = 1$ (resp. $\chi (p) = -1$) and $\chi \vert _{\F^{\times}_{p^n} \times U^1_E}$ has order $t$.
In particular, $\chi$ is a character as in Example \ref{eje1} and the pair $(E/\Q_p, \chi)$ is admissible. 

Let $\ell$ be a prime distinct from $p$ and $t$. From now on, we will fix an isomorphism $\iota : \overline{\Q}_\ell \cong \C$, which allows us to compare $\overline{\Q}_\ell$-valued characters with $\C$-valued ones.  As we pointed out previously, by local class field theory, we can regard $\chi$ as a character of $W_E$, or in fact as a character of $G_E$. Then, by the discussion in the previous section, we have a $\GL_n(\overline{\Q}_\ell)$-valued, $n$-dimensional, tamely ramified, self-dual, irreducible Galois representation
\[
\rho_{\chi} = \Ind^{G_{\Q_p}}_{G_E}\chi
\] 
associated to $(E/\Q_p, \chi)$. When $\chi$ is of $O$-type (resp. $S$-type) at $p$ of order $t$, it can be proven that $\rho_{\chi}$ is orthogonal (resp. symplectic) in the sense that it can be conjugated to take values in $\SO_n(\overline{\Q}_\ell)$ (resp. $\Sp_n(\overline{\Q}_\ell)$). Moreover, if $\alpha: G_{\Q_p} \rightarrow \overline{\Q}^\times_\ell$ is an unramified character, then the residual representation $\overline{\rho}_{\chi} \otimes \overline{\alpha}$ is irreducible and tensor-indecomposable. See Section 5 of \cite{Ze19} for details. 

Henceforth, we will assume that $\chi$ is of $O$-type at $p$ of order $t$ because the characters of $S$-type have been studied extensively in \cite{KLS08} and \cite{ADW16}. Let $\Gamma_t$ be a non-abelian homomorphic image of an extension of $\Z / n\Z$ by $\Z/t\Z$ such that $\Z/n\Z$ acts faithfully on $\Z/t\Z$. For example, $\Gamma_t$ can be the image of the residual representation $\overline{\rho}_{\chi}$ of $\rho_{\chi}$.
On the other hand, let $\Gamma$ be a group and $d$ be a positive integer. We define $\Gamma^d$ as the intersection of all normal subgroups of $\Gamma$ of index at most $d$.  Then we have the following result which gives us information about $\Gamma$ when $\Gamma^d$ contains $\Gamma_t$.

\begin{theo}\label{kls22}
There exist constants $d(n)$ and $t(n)$ depending  only on $n$ such that, if $d > d(n)$ is an integer, $t>t(n)$ and $\ell$ are distinct primes, and $\Gamma \subseteq \GL_n(\overline{\F}_\ell)$ is a finite group such that $\Gamma^d$ contains $\Gamma_t$, then there exist $g \in \GL_n(\overline{\F}_{\ell})$ and a positive integer $k$ such that $g^{-1}\Gamma g$ is a group containing one of the following groups: $\SL_n(\F_{\ell^k})$, $\SU_n(\F_{\ell^k})$, $\Sp_n(\F_{\ell^k})$ or $\Omega_n^{\pm}(\F_{\ell^k})$.
\end{theo}

\begin{proof}
The proof is the same as that of Theorem 2.2 of \cite{KLS08}, since it only depends on the metacyclic group structure of $\Gamma_t$ and the faithful action of $\Z/n\Z$ on $\Z /t\Z$. 
\end{proof}

\begin{coro}\label{kls26}
Let $n\geq 8$ and $\ell$ be an odd prime. Under the hypothesis of the previous theorem, if $\Gamma \subseteq \GO_n(\overline{\F}_{\ell})$, then there exists $g \in \GL_n(\overline{\F}_{\ell})$ such that $g^{-1} \Gamma g$ contains $\Omega_n^\pm (\F_{\ell^k})$ for some positive integer $k$.
\end{coro}

\begin{proof}
The proof is adapted from Corollary 2.6 of \cite{KLS08} where the symplectic case is dealt with. Assume that $g^{-1}\Gamma g$ contains $\SL_n(\F_{\ell^k})$, $\SU_n(\F_{\ell^k})$ or $\Sp_n (\F_{\ell^k})$. As $\Gamma \subseteq \GO_n(\overline{\F}_{\ell})$, one of these groups has an $n$-dimensional symmetric representation. Then, by Steinberg's theorem, we have that the algebraic group $\SL_n$ or $\Sp_n$ has a non-trivial self-dual $n$-dimensional representation defined over $\overline{\F}_{\ell}$ which maps the fixed points of a Frobenius map into $\SO_n(\overline{\F}_\ell)$. However, $\SL_n$ has no non-trivial self-dual representation of dimension $n$ when $n >2$ and as $\ell$ is odd, an irreducible $n$-dimensional representation of $\Sp_n (\F_{\ell^k})$ cannot preserve a symmetric form since it already preserves a symplectic form.
Then by Theorem \ref{kls22} $g^{-1}\Gamma g$ contains $\Omega^{\pm}_n(\F_{\ell^k})$.
\end{proof}

\section{Maximally induced representations and the inverse Galois problem}\label{se3}

Let $n$ be a positive even integer, and $p$, $t$, $\ell$, $\chi$ and $\rho_\chi$ as in the previous section. As before, we assume that $\chi$ is of $O$-type at $p$ of order $t$. We say that a Galois representation 
\[
\rho_\ell: G_\Q \longrightarrow \GO_n(\overline{\Q}_\ell)
\]
is \emph{maximally induced} of $O$-type at $p$ of order $t$ if the restriction of $\rho_\ell$ to a decomposition group $D_p$ at $p$ is equivalent to $\rho_\chi \otimes \alpha$ for some unramified character $\alpha : G_{\Q_p} \rightarrow \overline{\Q}_\ell^{\times}$.
The following result states that, for an appropriate couple of primes $(p, t)$, the residual image of a maximally induced representation is large. 

\begin{theo}\label{adz}
Let $n \geq 8$ be an even integer, and $d(n)$, $t(n)$ be constants as in Theorem \ref{kls22}. Let $\ell$ be an odd prime and $K$ be the compositum of all number fields of degree at most $d(n)+1$ which are unramified outside $\{ \ell, \infty \}$. 
Let $(p, t)$ be a couple of primes satisfying that $p > \ell$ splits completely in $K$, $t \equiv 1 \mod n$, $t > \max \{ d(n)+1, t(n),\ell \}$ and the order of $p$ modulo $t$ is $n$.
Let 
\[
\rho_\ell: G_\Q \longrightarrow \GO_n(\overline{\Q}_\ell)
\] 
be a maximally induced Galois representation of $O$-type at $p$ of order $t$ which is  unramified outside $\{p, \ell \}$. Then, the image of $\overline{\rho}_\ell$ contains $\Omega_n^\pm (\F_{\ell^k})$ for some positive integer $k$.
\end{theo}

\begin{proof}
The proof is adapted from Proposition 5.5 of \cite{AD} where a character of $S$-type is used.
Let $\Gamma$ be the image of $\overline{\rho}_\ell$ and $H$ be a normal subgroup of $\Gamma$ of index at most $d(n)+1$. Associated to $H$, we have a Galois extension $L/\Q$ of degree at most $d(n)+1$. Moreover, by the ramification of $\rho_\ell$, we have that $L/\Q$ is unramified
outside $\{p, \ell, \infty \}$.

As $\rho_\ell$ is maximally induced of $O$-type at $p$ of order $t$, we have that the restriction $\overline{\rho}_\ell \vert _{D_p}$ is equivalent to
$\overline{\rho}_\chi \otimes \overline{\alpha}$, for some unramified character $\alpha$, and the image of the inertia $\overline{\rho}_\ell(I_p)$ has order $t$. Then, as $t > d(n)+1$, we can conclude that $L/\Q$ is unramified at $p$. In particular, $L$ is unramified outside $\{ \ell, \infty\}$, so $L/\Q$ is a subextension of $K/\Q$. Moreover, by the choice of $p$, it is split in $L$. Thus, $\overline{\rho}_\ell(D_p)$ (which is isomorphic to a non-abelian homomorphic image of an extension of $\Z / n\Z$ by $\Z/t\Z$ such that $\Z/n\Z$ acts faithfully on $\Z/t\Z$) is contained in $H$.
Therefore, we can conclude that $\Gamma$ satisfies the conditions of Theorem \ref{kls22}, taking $d = d(n)+1$. Finally, as $\Gamma$ is contained in $\GO_n(\overline{\F}_\ell)$ our result follows from Corollary \ref{kls26}.
\end{proof}

In order to use the previous result to prove our result on the inverse Galois problem, we need to find a source of Galois representations satisfying the hypothesis in Theorem \ref{adz}. In our case, such a source will be certain automorphic representations of $\GL_n(\A_\Q)$. More precisely, let $\pi$ be a RAESDC (regular algebraic, essentially self-dual, cuspidal) automorphic representation of $\GL_n(\A_{\Q})$ unramified outside a finite set of primes $S$. Then, by the work of Caraiani, Chenevier, Clozel, Harris, Kottwitz, Shin, Taylor and several others; we have that, for any prime $\ell$, there exists a semi-simple Galois representation
\[
\rho_{\pi,\ell}: G_{\Q} \longrightarrow \GL_n(\overline{\Q}_\ell)
\]
unramified outside $S \cup \{ \ell \}$, compatible with the local Langlands correspondence. In particular, as $\pi$ is essentially self-dual, the image of $\rho_{\pi,\ell}$ is contained in $\GO_n(\overline{\Q}_\ell)$ or $\GSp_n(\overline{\Q}_\ell)$. See \cite[Section 2.1]{BLGGT14} and \cite[Section 3]{Ze19} for details and references.

\paragraph{Proof of Theorem \ref{impo}}
Let $\ell$ be an odd prime and $(p,t)$ be a couple of primes satisfying the hypothesis of Theorem \ref{adz}. The existence of such a couple of primes is guaranteed by Chevotarev's Density Theorem as in Lemma 6.3 of \cite{Ze19}.
Let $(E/\Q_p, \chi)$ be an admissible pair (as in Section \ref{se2}), with $\chi$ of $O$-type at $p$ of order $t$. As we remark in Section \ref{se1}, (by local Langlands correspondence) associated to $(E/\Q_p, \chi)$, there is a self-dual supercuspidal representation $\pi_{\mu \chi}$ of $\GL_n(\Q_p)$ of depth zero. 

Let $n=2m \geq 8$. Following Section 7 of \cite{Ze19}, we can construct a RAESDC automorphic representation $\pi = \otimes_v \pi_v $ of $\GL_n(\A_\Q)$ unramified outside $\{ p \}$, such that the local component $\pi_p$ of $\pi$ at the prime $p$ is $\pi_{\mu \chi}$. 
We remark that in \emph{loc. cit.}, a parity restriction on $m$ was imposed due to the fact that the split orthogonal group $\SO(m,m)$ has discrete series if and only if $m$ is even. However, the construction in \cite{Ze19} (which follows from the results of Arthur \cite{Art13} and Shin \cite{Shi12}, that also work for quasi-split orthogonal groups) can be extended to $m$ odd, by considering the quasi-split group $\SO(m+1,m-1)$ which has discrete series\footnote{This follows from Harish-Chandra's criterion for the existence of discrete series representations for the non-exceptional real Lie groups. See Chapter XII of \cite{Kna86} and its bibliographic notes.}.

Then, for all $\ell$, there exists a semi-simple Galois representation $\rho_{\pi,\ell}: G_{\Q} \rightarrow \GL_n(\overline{\Q}_\ell)$ unramified outside $\{p, \ell\}$ and maximally induced of $O$-type at $p$ of order $t$. In fact, as $\overline{\rho}_{\chi}$ is irreducible and orthogonal, $\overline{\rho}_{\pi,\ell}$ is irreducible and its image is contained in $\GO_n(\overline{\Q}_\ell)$. Thus, by Theorem \ref{adz}, we have that $\overline{\rho}_{\pi,\ell}$ contains $\Omega_n^\pm (\F_{\ell^k})$ for some positive integer $k$.  Consequently, the image of $\overline{\rho}^{\proj}_\ell$ (the projectivization of $\overline{\rho}_\ell$) is one of the following groups: 
\[
\POM^\pm_n(\F_{\ell^s}), \; \PSO^\pm_n(\F_{\ell^s}), \; \PO_n^\pm(\F_{\ell^s}) , \; \PGO^\pm_n(\F_{\ell^s})
\]
for some positive integer $s$, which implies that such group can be realized as a Galois group over $\Q$.

Finally, by Chebotarev's Density Theorem, there are infinitely many ways to choose the couple of primes $(p, t)$, and we can construct infinitely many RAESDC automorphic representations $\{ \pi_i \}_{i\in \N}$ of $\GL_n(\A_\Q)$ as above. Hence, there exists a family of Galois representations $\{ \rho_{\pi_i, \ell} \}_{i\in \N}$ such that the size of the image of $\rho^{\proj}_{\pi_i, \ell}$ is unbounded for running $i$, because we can choose $t$ as large as we please so that elements of larger and larger orders appear in the inertia images. This concludes our proof. 

\paragraph{Acknowledgments:}
I would like to thank Luis Dieulefait for many enlightening discussions and his encouragement to work on this project. I thank Luis Lomel\'i for useful discussions about automorphic representations. I also thank Andrew Odesky for useful comments on a previous version. Finally, I want to give special thanks to the referee, whose suggestions and comments have greatly improved the presentation of this paper.  In particular, his/her questions helped me remove a parity restriction imposed in a previous version of this manuscript.

The author was supported by CONICYT Proyecto FONDECYT Postdoctorado No. 3190474.

\end{document}